\documentclass[a4paper,12pt,reqno]{article}

\usepackage{amsmath}
\usepackage[hmargin=1in,vmargin=1in,nohead]{geometry}

\usepackage{amssymb}
\usepackage{amsfonts}
\usepackage{amsthm}
\usepackage[french,english]{babel}
\usepackage{enumerate}
\usepackage{graphicx}
\usepackage{a4wide}
\usepackage[latin1]{inputenc}  % accents 8 bits dans le source
\usepackage[T1]{fontenc}
\usepackage{bbm}
\usepackage{xspace}

\usepackage{verbatim} %pour avoir la commande \begin{comment} \end{comment} :)

\RequirePackage[OT1]{fontenc}

\theoremstyle{plain}
\newtheorem{theorem}{Theorem}

\newtheorem{corollary}{Corollary}

\newtheorem{definition}{Definition}

\newtheorem{lemma}{Lemma}
\newtheorem{proposition}{Proposition}

\numberwithin{equation}{section}

\newcommand{\Real}{\mathbb R}
\newcommand{\N}{\mathbb N}
\newcommand{\Z}{\mathbb Z}

\newcommand{\Pro}{\mathbf{P} }

\newcommand{\Prob}{\mathbb P}
\newcommand{\Probc}{\Prob^c}
\newcommand{\Probt}{\widetilde{\Prob}}
\newcommand{\Psg}{\mathrm{P}}
\newcommand{\eps}{\varepsilon}
\newcommand{\To}{\longrightarrow}
\newcommand{\Inv}{\frac{1}}
\newcommand{\n}{\mathbf{n}}
\newcommand{\nbis}{\overline \n}
\newcommand{\nti}{\tilde \n}

\newcommand{\E}{\mathcal E}
\newcommand{\un}{\mathbbm 1}

\newcommand{\Tr}{\zeta}

\newcommand{\Vr}{V}

\newcommand{\Tri}{\Tr_\infty}
\newcommand{\Pup}{\widetilde {\mathbf{P}}}
\newcommand{\Probup}{\widetilde {\Prob}}
\newcommand{\Probupzero}{\Probup_{0^+}}
\newcommand{\Pupz}{\Probupzero}
\newcommand{\Proup} {\widetilde \Pro}
\newcommand{\muup} {\mu}%{\widetilde {\mu}}
\newcommand{\overup} {\overshoot}%{\widetilde \overshoot}

\newcommand{\B}{\mathrm{B}}

\newcommand{\F}{\mathfrak F}
\newcommand{\de}{\mathrm{d}}
\newcommand{\ccv}{\exp(-\pi / \sqrt 3)}
\newcommand{\overshoot}{m}
\newcommand{\ccr}{c_{crit}}

\newcommand{\Cbb}{\Ccal}
\newcommand{\Ccal}{\mathcal{C}}

\newcommand{\espace}{\vspace{.4cm}}
\newcommand{\Pcz}{\Prob_0^r} %processus ressuscité
\newcommand{\Ae}{A^\eps}
\newcommand{\Ti} {\widetilde}
\newcommand{\Thetasp}{\Theta^{sp}}

\newcommand{\citejacobIHP} {\cite{jacob2010}\xspace}
\newcommand{\citejacobRLPI} {\cite{jacobRLP1}\xspace}

\date{}
\title{  \bf Langevin process reflected on a partially elastic boundary II}
\author{ \bf Emmanuel Jacob\footnote{
{\bf email.} emmanuel.jacob@normalesup.org \newline {\null \quad
\;\;\:{\bf website.} http://www.proba.jussieu.fr/pageperso/jacob}} }

\begin{document}

\maketitle{}
                        { \center
\emph{Laboratoire de Probabilit\'es et Mod\`eles Al\'eatoires} \\
\emph {Universit\'e Pierre et Marie Curie} \\
\emph{4 place Jussieu, 75005 Paris, France} \espace

\textbf{Abstract} \espace

\hspace{.8cm} \begin{minipage}{0.9\textwidth} A particle subject to
a white noise external forcing moves like a Langevin process.
Consider now that the particle is reflected at a boundary which
restores a portion $c$ of the incoming speed at each bounce. For $c$
strictly smaller than the critical value $\ccr=\exp(-\pi/\sqrt3)$,
the bounces of the reflected process accumulate in a finite time. We
show that nonetheless the particle is not necessarily absorbed after
this time. We define a ``resurrected'' reflected process as a
recurrent extension of the absorbed process, and study some of its
properties. We also prove that this resurrected reflected process is
the unique solution to the stochastic partial differential equation
describing the model. Our approach consists in defining the process
conditioned on never being absorbed, via an $h-$transform, and then
giving the It\=o excursion measure of the recurrent extension thanks
to a formula fairly similar to Imhof's relation.
\end{minipage}
                                }
\espace

 \vspace{9pt} \noindent {\bf Key words.}
{Langevin process}, {second order reflection}, {recurrent
extension}, {excursion measure}, {stochastic partial differential
equation}, {$h$-transform.}
\par \vspace{9pt}

\noindent {\bf A.M.S classification.} (MSC2010)\quad {60J50},
{60H15} \espace

\section{Introduction}

Consider a particle in a one-dimensional space, submitted to a white
noise external forcing. Its velocity is then well-defined and given
by a Brownian motion, while its position is given by a so-called
Langevin process. The Langevin process is non-Markov, therefore its
study is often based on that of the Kolmogorov process, which is
Markov. This Kolmogorov process is simply the two-dimensional
process, whose first coordinate is a Langevin process, and second
coordinate its derivative. We refer to Lachal \cite{Lachal03} for a
detailed account about it. Further, suppose that the particle is
constrained to stay in $[0,+\infty[$ by a boundary at 0
characterized by an elasticity coefficient $c\ge 0$. That is, the
boundary restores a portion $c$ of the incoming velocity at each
bounce, and the equation of motion that we consider is the
following:
$$ (SOR)\qquad \left\{ \begin{array}{ccl} %\label{eqmouvementII}
X_t &=& X_0 + \displaystyle \int_0^t \dot X_s\de s \\
 \\
\dot X_t &=& \dot X_0+ B_t - (1+c) \sum_{0< s \le t} \dot X_{s-}
\un_{X_s=0},
\end{array}
\right. $$ where $B$ is a standard Brownian motion and $(X_0,\dot
X_0)$ is called the initial or starting condition. This stochastic
partial differential equation is nice outside the point $(0,0)$.
Indeed, if the starting condition is different from $(0,0)$, there
is a simple pathwise construction of the solution to this equation
system, until time $\Tri$, the hitting time of $(0,0)$ for the
process $(X,\dot X)$. However there is a tough problem at $(0,0)$.
Indeed, there exists an old literature about a deterministic
analogue to theses equations, where the white noise force is
replaced by a deterministic force. See Ballard \cite{Ballard} for a
vast review. As early as in 1960, Bressan \cite{Bressan} pointed out
that multiple solutions may occur, even when the force is $\mathcal
C^\infty$. It appears that the introduction of a white noise allows
to get back a weak uniqueness result. We refer to \cite{SDE} (see
also \cite{reflecting}, \cite{jacob2010}) for the particular case
$c=0$.

\espace

In \citejacobRLPI, we have shown for $c>0$ the existence of two
different regimes, the critical elasticity being $\ccr:=\ccv$. It is
critical in the sense that when the starting condition is different
from $(0,0)$, then we have $\Tri=+\infty$ almost surely if $c\ge
\ccr$, and $\Tri<+\infty$ almost surely if $c<\ccr$. Further, we
studied the super-critical and the critical regimes. In this paper,
we study the sub-critical regime $c<\ccr$. The finite time $\Tri$
corresponds to an accumulation of bounces in a finite time. We write
$\Prob_{x,u}^c$ for the law of the reflected Kolmogorov process,
with starting condition $(x,u)\ne (0,0)$, elasticity coefficient
$c$, and \emph{killed at time $\Tri$}. It is the unique strong
solution to $(SOR)$ equations, up to time $\Tri$. We also write
$\Psg_t^c$ for the associated semigroup. We will devote ourselves to
prove the existence of a unique recurrent extension to this process
that leaves $(0,0)$ continuously. Moreover, we will prove that this
extension gives the unique solution, in the weak sense, to $(SOR)$
equations.

We point out that this model was encountered by Bect in his thesis
(\cite{Bect07}, section III.4.B). He observed the existence of the
critical elasticity and asked several questions on the different
regimes. We answer to all of them.

\espace

In this work we will be largely inspired by a paper of Rivero
\cite{rivero04}, in which he studies the recurrent extensions of a
self-similar Markov process with semigroup $\Psg_t$. Briefly, first,
he recalls that recurrent extensions are equivalent to excursion
measures compatible with $\Psg_t$, thanks to It\=o's program.
 Then a change of probability allows him to define the Markov process conditioned on never hitting
 $0$, where this conditioning is in the sense of Doob, via an $h-$transform. An inverse $h-$transform
 on the Markov process conditioned on never hitting zero \emph{and starting from 0}
 then gives the construction of the excursion measure.

%\noindent
 We will not recall it at each step throughout the paper,
but a lot of parallels can be made. However, it is a two-dimensional
Markov process that we consider here. Further, its study will rely
on an underlying random walk $(S_n)_{n\in\N}$ constructed from the
velocities at bouncing times.

\espace

In the Preliminaries, we introduce this random walk and use it to
estimate the tail of the variable $\Tri$ under $\Prob_{0,1}^c$. In
the Section~\ref{section_conditioned}, we introduce a change of
probability, via an $h-$transform, to define $\Probup_{x,u}$, law of
a process which can be viewed as the reflected Kolmogorov process
conditioned on never being killed. We then show in
Subsection~\ref{section_Pupz} that this law has a weak limit $\Pupz$
when $(x,u)$ goes to $(0,0)$, using the same method that was used in
\citejacobRLPI to show that for $c>\ccr$, the laws $\Prob_{0,u}^c$
have the weak limit $\Prob_{0+}^c$ when $u$ goes to zero. All this
section can be seen as a long digression to prepare the construction
of the excursion measure in Section~\ref{section_Resurrect}. This
excursion measure is defined by a formula similar to Imhof's
relation (see \cite{Imhof}), connecting the excursion measure of
Brownian motion and the law of a Bessel(3) process. But our formula
involves the law $\Prob_{0+}^c$ and determines the unique excursion
measure compatible with the semigroup $\Psg_t^c$ . We call
\emph{resurrected Kolmogorov process} the corresponding recurrent
extension. Finally, we prove that this is the (weakly) unique
solution to $(SOR)$ equations when the starting condition is
$(0,0)$.

\section{Preliminaries}
We largely use the same notations as in \citejacobRLPI. For the sake
of simplicity, we use the same notation (say $P$) for a probability
measure and for the expectation under this measure. We will even
authorize ourselves to write $P(f,A)$ for the quantity $P(f \un_A)$,
when $f$ is a measurable functional and $A$ an event. We introduce
$D= (\{0\}\times \Real_+^* )\cup (\Real_+^*\times \Real)$ and
$D^0:=D \cup\{(0,0)\}$. Our working space is $\Cbb$, the space of
càdlàg trajectories $(x,\dot x): [0,\infty) \to D^0$, which satisfy
$$x(t) = x(0)+ \displaystyle \int_0^t \dot x(s)\de s.$$
That space is endowed with the $\sigma-$algebra generated by the
coordinate maps and with the topology induced by  the following
injection:
$$ \begin{array}{ccc}
\Cbb &\to& \Real_+ \times \mathbb D \\
(x,\dot x) & \mapsto & \big(x(0), \dot x\big),
\end{array}
$$
where $\mathbb D$ is the space of càdlàg trajectories on $\Real_+$,
equipped with Skorohod topology.

We denote by $(X,\dot X)$ the canonical process and by
$(\F_t,t\ge0)$ its natural filtration, satisfying the usual
conditions of right continuity and completeness. For an initial
condition $(x,u)\in D$, the $(SOR)$ equations
 $$ \left\{ \begin{array}{ccl} %\label{eqmouvementII}
X_t &=& x + \displaystyle \int_0^t \dot X_s\de s \\
 \\
\dot X_t &=& u + B_t - (1+c) \sum_{0< s \le t} \dot X_{s-}
\un_{X_s=0}
\end{array}
\right. $$ have a unique solution, at least up to the random time
$$\Tri:= \inf \{t>0, X_t=0,\dot X_t=0\}.$$
We call (killed) reflected Kolmogorov process this solution killed
at time $\Tri$, and write $\Prob_{x,u}^c$ for its law. It is Markov.
We also call reflected Langevin process the first coordinate of this
process, which is no longer Markov.

Call $\Tr_1$ the first hitting time of zero for the reflected
Langevin process $X$, that is $\Tr_1:=\inf \{t>0, X_t=0\}$. More
generally, the sequence of the successive hitting times of zero
$(\Tr_n)_{n\ge1}$ is defined recursively by $\Tr_{n+1}:=\inf
\{t>\Tr_n, X_t=0\}$. We write $(\Vr_n)_{n\ge1}:=(\dot
X_{\Tr_n})_{n\ge1}$ for the sequence of the velocities of the
process at these hitting times. That means outgoing velocities, as
we are dealing with right-continuous processes. Finally, when the
starting position is $x=0$, we will simply write $\Probc_u$ for
$\Probc_{0,u}$, and we will also define $\Tr_0=0$ and $\Vr_0=\dot
X_0$. We insist on the fact that in each case the starting condition
$(x,u)$ is different from $(0,0)$. Then it is not difficult to see
that $\Tri$ coincides almost surely with $\sup \Tr_n$. But we can
say much more.

\espace The sequence $\left( \dfrac {\Tr_{n+1}- \Tr_n} {\Vr_n^2},
\dfrac {\Vr_{n+1}} {\Vr_n} \right)_{n\ge0}$ is i.i.d. and of law
independent of $u$, which can be deduced from the following density:
\begin{equation} \label{loijointe_VrTrII}
\Inv {\de s \de v} \Prob_1^c\left( ( {\Tr_{1}}, {\Vr_{1}} /c) \in
(\de s, \de v)  \right) = \frac {3 v} {\pi \sqrt2 s^2}
\exp(-2\frac{v^2-v+1}{s}) \int_0^{4v/ s} e^{-\frac {3 \theta} 2}
\frac {\de \theta} {\sqrt {\pi \theta}},
\end{equation}
given by McKean \cite{McKean}. The second marginal of this density
is
\begin{equation} \label{loi_VrI}
\Prob_1^c( {\Vr_{1}} /c \in \de v)= \frac 3 {2\pi} \frac {v^{\frac 3
2}} {1+v^3} \de v.
\end{equation}
In particular, the sequence $S_n:=\ln(\Vr_n)$ is a random walk, with
drift
$$\Prob_1^{c}(S_1 - S_0) = \ln(c)+\frac \pi {\sqrt 3},$$
which is zero for the critical value $\ccr=\ccv$. In this paper we
lie in the subcritical case $c< \ccr$, when the drift is negative. A
thorough study allows to not only deduce the finiteness of $\Tri$,
but also estimate its tail.

\begin{lemma} \label{lemme_queueTri}
We have
\begin{equation}\label{moments_Vr}
    \Prob_1^c\left( \Vr_1^x\right) = \frac {c^x} {2 \cos(\frac {x+1} 3 \pi)} \textrm{ for } x< 1/2.
\end{equation}
There exists a unique $k=k(c)$ in $(0,1/4)$ such that $
\Prob_1^c\left( \Vr_1^{2k}\right)=1$, and
 \begin{equation} \label{goldie}
\Prob_1^c(\Tri > t) \underset{t\to \infty}{\sim} C_1 t^{-k},
\end{equation}
where $C_1=C_1(c) \in (0,\infty)$ is a constant depending only on
$c$, given by
\begin{equation} \label{valeur_C_1}
C_1 = \frac  {\Prob_1^c \left(\Tri^k - (\Tri- \Tr_1)^k \right)} {k
\Prob_1^c( \Vr_1^{2k} \ln(\Vr_1^2)) }.
\end{equation}
\end{lemma}
 In other words, $k(c)$ is given implicitly as the unique solution in
$]0,\Inv 4]$ of the equation
\begin{equation} \label{implicitk}
c = \left[ 2 \cos\left(\frac {2k+1} 3 \pi \right)\right]^{\Inv {2k}}.%, \quad 0< k< \Inv 4.
\end{equation}
The upper bound $1/4$ stems from the fact that $\Prob_1^c\left(
\Vr_1^{2k}\right)$ becomes infinite for $k=1/4$. The value of $k(c)$
converges to $1/4$ when $c$ goes to 0, and to $0$ when $c$ goes to
$\ccr$, as illustrated by Figure~\ref{figure_k_c}. We may notice
that Formula~\eqref{goldie} remains true for $c=0$ and $k=1/4$ (and
for $c=\ccr$ and $k=0$, in a certain sense).

\begin{figure}[!ht]
\centering %\noindent
\includegraphics[width=\textwidth]{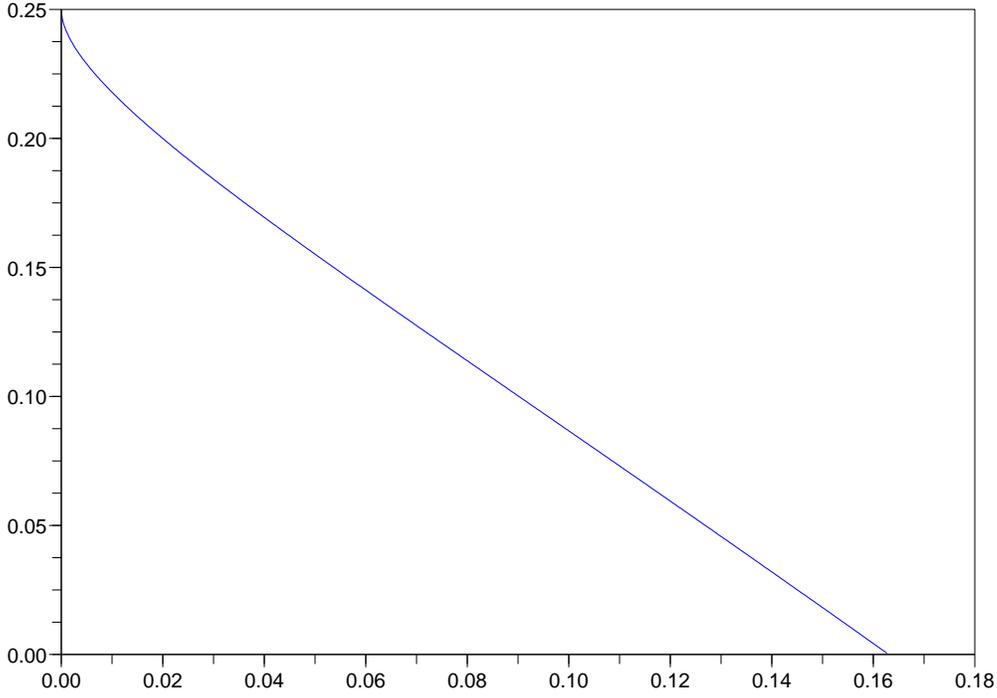}
\caption{Graph of the exponent k(c)}
 \label{figure_k_c}
 \end{figure}

\begin{proof}
Formula \eqref{moments_Vr} is not new. For the convenience of the
reader, we still provide the following calculation. From
Formula~\eqref{loi_VrI}, it follows, for $x<1/2$,
\begin{eqnarray*}
 \Prob_1^c\left(\left({\Vr_1}/c \right)^x\right) &=& \frac 3 {2 \pi} \int_0^\infty \frac {t^{x+ 3/2}} {1+t^3} \de t
 = \Inv {2 \pi} \int_0^\infty \frac {t^{\frac x 3 - \Inv 6}} {1+t} \de t.
\end{eqnarray*}
Note $\cos(\frac {x+1} 3 \pi) = \sin(\pi y)$, where $y= \frac x 3 +
\frac 5 6$. Using the variable $y$, which belongs to $(0,1)$,
Equation \eqref{moments_Vr} becomes
$$ \int_0^{\infty} \frac {t^{y-1}} {1+t} \de t = \frac \pi {\sin(\pi y)},$$
and follows from:
\begin{eqnarray*}
\int_0^{\infty} \frac {t^{y-1}} {1+t} \de t &=& \int_0^1 t^y (1-t)^{1-y} \de t \\
&=& \B(y,1-y) \\
&=& \frac {\Gamma(y) \Gamma(1-y)} {\Gamma(1)} \\
&=& \frac \pi {\sin(\pi y)}.
\end{eqnarray*}
where $\B$ and $\Gamma$ are the usual Beta and Gamma function,
respectively.

 Now, the function $x \mapsto \Prob_1^c\left( \Vr_1^x\right)$ is convex,
takes value 1 at $x=0$ and becomes infinite at $x= 1/2$. Its
derivative at 0 is equal to $\Prob_1^c(S_1-S_0)<0.$ We deduce that
there is indeed a unique $k(c)$ in $(0,\Inv 4)$ such that
$\Prob_1^c\left( \Vr_1^{2k}\right) = 1.$ \espace

Estimate \eqref{goldie} will appear as a particular case of an
``implicit renewal theory'' result of Goldie \cite{Goldie}. Let us
express $\Tri$ as the series:
\begin{eqnarray*}
\Tri &=& \sum_{n=1}^\infty \frac {\Tr_n - \Tr_{n-1}} {\Vr_{n-1}^2}
\Vr_{n-1}^2 ,
\end{eqnarray*}
with $\Vr_n^2:=\Vr_1^2 \frac {\Vr_2^2} {\Vr_1^2} \cdot\cdot\cdot
\frac {\Vr_n^2} {\Vr_{n-1}^2}$, and where $\left( \dfrac {\Tr_n-
\Tr_{n-1}} {\Vr_{n-1}^2}, \dfrac {\Vr_n^2} {\Vr_{n-1}^2}
\right)_{n\ge1}$ is i.i.d. We lie in the setting of Section~4 of
Goldie's paper \cite{Goldie}, and can apply its Theorem~(4.1).
Indeed, all the following conditions are satisfied:
$$\Prob_1^c( \Vr_1^{2k})=1,$$
$$\Prob_1^c( \Vr_1^{2k} \ln(\Vr_1^2))<\infty,$$
$$\Prob_1^c( \Tr_1^k )< \infty,$$
the last one being a consequence of the inequality $k<1/4$ and of
the
following estimate of the queue of the variable $\Tr_1$,%a result of
%Goldman already stated in \cite{jacobRLP1}, namely
\begin{equation} \label{queue_TrrII} \Prob_1^c(\Tr_1>t) \underset{t\to \infty} {\sim} c' t^{-\Inv
4},
\end{equation}
which was already pointed out in Lemma~1 in~\citejacobRLPI. All this
is enough to apply the theorem of Goldie and deduce the requested
result, namely
$$ \Prob_1^c(\Tri > t) \underset{t\to \infty} {\sim} C_1 t^{-k}, $$
where $C_1$ is the constant defined by~\eqref{valeur_C_1}, and
belongs to $]0,\infty[$.
\end{proof}

Next section is devoted to the definition and study of the reflected
Kolmogorov process, \emph{conditioned on never hitting $(0,0)$}.
This process will be of great use for studying the recurrent
extensions of the reflected Kolmogorov process in
Section~\ref{section_Resurrect}.

\section{The reflected Kolmogorov process conditioned on never hitting $(0,0)$}
\label{section_conditioned}
\subsection{Definition via an $h-$transform}%}of the reflected process conditioned on never hitting $(0,0)$}

Recall that under $\Prob_1^c$, the sequence $(S_n)_{n\ge0}=
(\ln(\Vr_n))_{n\ge0}$ is a random walk starting from 0, and write
$\Pro_0$ for its law.
 The important fact $\Prob_1^c(\Vr_1^{2k})=1$
implies $\Prob_1^c(\Vr_n^{2k})=1$ for any $n>0$, and can be
rewritten $\Pro_0(\theta^{S_n})=1$, with $\theta:= \exp(2k)$.

The sequence $\theta^{S_n}$ being a martingale, we introduce the
change of probability $$\Pup_0(S_n \in \de t)= \theta^t \Pro_0(S_n
\in \de t).$$ Under $\Pup_0$, $(S_n)_{n\ge 0}$ becomes a random walk
drifting to $+\infty$. Informally, it can be viewed as being the law
of the random walk $S_n$ under $\Pro_0$ conditioned on hitting
arbitrary high levels.

There is a corresponding change of probability for the reflected
Kolmogorov process and its law $\Prob_1^c$. We introduce the law
$\Probup_1$ determined by
$$ \Probup_1(A \un_{\Tr_n>T}) = \Probc_1(A \un_{\Tr_n>T} \Prob_1^c(\Vr_n^{2k}| \F_T)),$$
for any $n>0$, stopping-time $T$ and $A \in \F_T$. By the strong
Markov property we have
$$\Prob_1^c(\Vr_n^{2k}| \F_T)= \Prob_{X_T,\dot X_T}^c( \Vr_1^{2k}) \qquad \text{on the event } \{\Tr_n>T\},$$
so that there is the identity
$$ \Probup_1(A \un_{\Tr_n>T}) = \Probc_1(A \un_{\Tr_n>T} H(X_T,\dot X_T)),$$
where we have written
$$H(x,u) :=\Prob_{x,u}^c( \Vr_1^{2k}).$$
Note that $H(0,u)= u^{2k}$. Letting $n$ go to infinity, we get:
$$ \Probup_1(A \un_{\Tri>T}) = \Probc_1(A \un_{\Tri>T} H(X_T,\dot X_T)).$$
We have $H(0,1)=1$, the function $H$ is harmonic for the semigroup
of the reflected Kolmogorov process, and the process $\Probup_1$ is
the $h-$transform of $\Prob_1^c$, in the sense of Doob.

Under $\Probup_1$, the law of the sequence $(S_n)_{n\ge 0}$ is
$\Pup_0$, thus this sequence is diverging to $+\infty$, and as a
consequence the time $\Tri$ is infinite $\Probup_1-$almost surely.
The term $\un_{\Tri>T}$ in $\Probup_1(A \un_{\Tri>T})$ is thus
unnecessary. We may now give a more general definition of this
change of probability, as an $h-$transform, for any starting
position $(x,u)$.
\begin{definition} \label{def_Pup} %Let us write $H(x,u)=\int h_v(x,u) v^{2k} \de v.$
The reflected Kolmogorov process \emph{conditioned on never hitting
$(0,0)$} is the Markov process given by its law $\Probup_{x,u}$, for
any starting condition $(x,u) \in D$, which is the unique measure
such that for every stopping-time $T$ we have
\begin{equation} \label{eq_def_Pup}
\Probup_{x,u}(A) = \Inv {H(x,u)} \Prob_{x,u}^c(A H(X_T,\dot X_T), T<
\Tri),
\end{equation}
 for any $A \in \F_T$. We write $\widetilde \Psg_t$ its associated
 semigroup, and we also write $\Probup_u$ for $\Probup_{0,u}$.
\end{definition}
This denomination is justified by the following proposition.
\begin{proposition}
 For any $(x,u) \in D$ and $t>0$, we have
\begin{equation} \label{neverhitting}
\Probup_{x,u}(A) = \lim_{s\to \infty} \Probc_{x,u}(A|\Tri>s),
\end{equation}
for any $A \in \F_t$.
\end{proposition}
We stress that in \cite{rivero04}, Proposition 2, Rivero defines in
a similar way the self-similar Markov process conditioned on never
hitting 0. Incidentally, you can find in \citejacobIHP a thorough
study of other $h-$transforms regarding the Kolmogorov process
killed at time $\Tr_1$.

In order to get Formula~\eqref{neverhitting}, we first prove the
following lemma, which is a slight improvement of \eqref{goldie}:
\begin{lemma} \label{lemme_queuebis}
For any $(x,u) \in D$,
\begin{equation} \label{goldiegeneralise}
s^k \Probc_{x,u}(\Tri > s) \underset {s\to \infty} \To H(x,u) C_1.
\end{equation}
\end{lemma}
\begin{proof}
For $(x,u)=(0,1)$, this is  \eqref{goldie}. For $x=0$, the rescaling
invariance property yields immediately
$$ s^k \Probc_{0,u}(\Tri > s)= s^k \Probc_{0,1}(\Tri > s u^{-2}) \underset{s\to \infty} \To u^{2k} C_1 = H(0,u) C_1.$$
For $(x,u) \in D$, the Markov property at time $\Tr_1$ yields
\begin{eqnarray*}
s^k \Probc_{x,u}(\Tri > s) &=& \Probc_{x,u}(s^k \Probc_{0,\Vr_1}(\Tri > s-\Tr_1)) \\
& \underset{s\to \infty} \To& \Probc_{x,u}(H(0,\Vr_1) C_1) = H(x,u)
C_1,
\end{eqnarray*}
where the convergence holds by dominated convergence. The lemma is
proved. \end{proof}

Formula~\eqref{neverhitting} then results from:% the following calculation:
\begin{eqnarray*}
\Probc_{x,u}(A|\Tri > s) &=& \Inv{\Probc_{x,u}(\Tri > s)} \Probc_{x,u}\left( A \Probc_{X_t,\dot X_t}(\Tri>s-t), \Tri>t\right) \\
 &\underset{s\to \infty} \To& \Inv {H(x,u)} \Probc_{x,u}\left(A H(X_T,\dot X_T),\Tri>t\right) \\
 &=& \ \Probup_{x,u}(A).
\end{eqnarray*}

%\begin{proposition}
%the law of $Z_s$ under $\Probup_{0,0}$
%is an entrance law for the semigroup $\Psg_t^\uparrow$.
%\end{proposition}

\subsection{Starting the conditioned process from $(0,0)$} \label{section_Pupz}
The study of the reflected Kolmogorov process conditioned on never
hitting $(0,0)$ will happen to be very similar to that of the
reflected Kolmogorov process in the supercritical case $c>\ccr$,
done in \citejacobRLPI. Observe the following similarities between
the laws $\Probup_u$, and $\Probc_u$ when $c>\ccr$: the sequence
$\left( \dfrac {\Tr_{n+1}- \Tr_n} {\Vr_n^2}, \dfrac {\Vr_{n+1}}
{\Vr_n} \right)_{n\ge0}$ is i.i.d., we know its law explicitly, and
the sequence $S_n=\ln(\Vr_n)$ is a random walk with positive drift.
It follows that a major part of \citejacobRLPI can be transcribed
\emph{mutatis mutandis}. In particular we will get a convergence
result for the probabilities $\Probup_u$ when $u$ goes to zero,
similar to Theorem~1 of \citejacobRLPI.

\espace

Under $\Probup_1$, the sequence $(S_n)_{n\ge0}$ is a random walk of
law $\Proup_0$. Write $\muup$ for its drift, that is the expectation
of its jump distribution, which is positive and finite. The
associated strictly ascending ladder height process $(H_n)_{n\ge0}$,
defined by $H_k = S_{n_k}$, where $n_0=0$ and
$n_k=\inf\{n>n_{k-1},S_n>S_{n_{k-1}}\}$, is a random walk with
positive jumps. Its jump distribution also has positive and finite
expectation $\muup_H \ge \muup$. The measure
\begin{equation}\label{def_overup}
\overup(\de y):= \Inv {\muup_H} \Proup_0(H_1>y) \de y.
\end{equation}
is the ``stationary law of the overshoot'', both for the random
walks $(S_n)_{n\ge0}$ and $(H_n)_{n\ge0}$. The following proposition
holds.

\begin{proposition} \label{proposition_theoRLP}
The family of probability measures $(\Probup_{x,u})_{(x,u) \in D}$
on $\Ccal$ has a weak limit when $(x,u) \to (0,0)$, which we denote
by $\Pupz$. More precisely, write $\tau_v$ for the instant of the
first bounce with speed greater than $v$, that is
$\tau_v:=\inf\{t>0, X_t=0, \dot X_t
> v\}.$ Then the law $\Pupz$ satisfies the following properties:
$$\begin{array}{ll}
(*) &  \left\{\begin{array}{l}
\displaystyle \lim_{v \to 0^+ } \tau_v = 0 \quad \text{almost surely}. \\%\qquad \Pupz\text{-almost surely} \\
\text{For any }u, v>0 \text{, and conditionally on }\dot X_{\tau_v}
= u\text{, the process } \\
(X_{\tau_v+ t}, \dot X_{\tau_v+ t})_{t\ge0} \text{ is independent of
} (X_s, \dot X_s)_{s< \tau_v} \text{ and has law }\Probup_u.
\end{array}
 \right. \\
 \\
(**)  & \text{For any }v>0, \text{ the law of }\ln(\dot X_{\tau_v}
/v) \text{ is } \overshoot.
 \end{array} $$
\end{proposition}
%\begin{proof}

In the proof of this proposition we can take $x=0$ and just prove
the convergence result for the laws $\Probup_u$ when $u \to 0+$. The
general result will follow as an application of the Markov property
at time $\Tr_1$.

The complete proof follows mainly the proof of Theorem~1 in
\citejacobRLPI and takes many pages. Here, the reader has three
choices. Skip this proof and go directly to next section about the
resurrected process. Or read the following for an overview of the
ideas of the proof, with details given only when significantly
different from that in \citejacobRLPI. Or, read \citejacobRLPI and
the following, if (s)he wants to get the complete proof.

\espace \espace

Call $T_y(S)$ the hitting time of $(y,\infty)$ for the random walk
$S$ starting from $x<y$. Call $\Proup_{\muup}$ the law of
$(S_n)_{n\ge0}$ obtained by taking $S_0$ and $(S_n - S_0)_{n\ge 0}$
independent, with law $\overup$ and $\Proup_0$, respectively. That
is, we allow the starting position to be nonconstant and distributed
according to $\muup$. A result of renewal theory states that the law
of the overshoot $(S_{n+T_y} - y)_{n\ge0}$ under $\Proup_x$, when
$x$ goes to $-\infty$, converges to $\Proup_m$. Now, for a process
indexed by $I$ an interval of $\Z$, we define a spatial translation
operator by $\Thetasp_y((S_n)_{n\in I}) = (S_{n+T_y} - y)_{n \in
{I-T_y}}$. We get that under $\Proup_x$ and when $x$ goes to
$-\infty$, the translated process $\Thetasp_y(S)$ converges to a
process called the ``spatially stationary random walk", a process
indexed by $\Z$ which is spatially stationary and whose restriction
to $\N$ is $\Proup_m$ (see \citejacobRLPI). We write $\Proup$ for
the law of this spatially stationary random walk.

\espace

There exists a link between the law $\Proup_x$ and the law
$\Probup_{e^x}$: the first one is the law of the underlying random
walk $(S_n)_{n\ge0} = (\ln \Vr_n)_{n\ge0}$ for a process $(X,\dot
X)$ following the second one. Now, in a very brief shortcut, we can
say that the law $\Proup$ is linked to a law written $\Pupz^*$. And
the convergence results of $\Proup_x \circ \Thetasp_y$ to $\Pro$
when $x\to -\infty$ provide convergence results of $\Probup_u$ to
$\Pupz^*$ when $u\to 0$.

However, this link is different, as the spatially stationary random
walk, of law $\Proup$, is a process indexed by $\Z$. The value $S_0$
is thus not equal to the logarithm of the velocity of the process at
time 0, but at time $\tau_1$ (recall that $\tau_1= \inf\{t>0, X_t=0,
\dot X_t\ge1\}$ is the instant of the first bounce with speed no
less than one). The sequence $(S_n)_{n\ge0}$ is then the sequence of
the logarithms of the velocities of the process at the bouncing
times, starting from that bounce. The sequence $(S_{-n})_{n\ge0}$ is
the sequence of the logarithms of the velocities of the process at
the bouncing times happening \emph{before} that bounce. \espace

The law $\Pupz^*$ is the law of a process indexed by $\Real_+^*$,
but we actually construct  it ``from the random time $\tau_1$''. In
order for the definition to be clean, we have to prove that the
random time $\tau_1$ is finite a.s. In \citejacobRLPI, we used the
fact that if $(\Tr_{1,k})_{k\ge0}$ is a sequence of i.i.d random
variables, with common law that of $\Tr_1$ under $\Probc_1$, then
for any $\eps>0$ there is almost surely only a finite number of
indexes $k$ such that $\ln(\Tr_{1,k})\ge \eps k.$ This was based on
Formula~\ref{queue_TrrII}, which, we recall, states
$$ \Prob_1^c(\Tr_1>t) \underset{t\to \infty} \sim c' t^{-\Inv 4},$$
where $c'$ is some positive constant. Here the same results holds
with replacing $\Probc_1$ by $\Probup_1$ and is a consequence from
the following lemma.
\begin{lemma}
We have  %But as from \cite{Goldman2} we know
\begin{equation} \label{queue_Trr_up} \Probup_1(\Tr_1>t) \underset{t\to \infty} \sim c' t^{k-\Inv
4},
\end{equation}
where $c'$ is some positive constant.
\end{lemma}
\begin{proof} %[Proof of the lemma]
From \eqref{eq_def_Pup} and \eqref{loijointe_VrTrII}, we get that
the density of $(\Tr_1,\Vr_1/c)$ under $\Probup_1$ is given by
\begin{eqnarray*}
 f(s,v):= \Inv {\de s \de v} \Probup_1((\Tr_1,\Vr_1/c) \in \de s \de v) &=& (cv)^{2 k} \frac {3 v} {\pi \sqrt2 s^2} \exp(-2\frac{v^2-v+1}{s}) \int_0^{\frac {4v} s} e^{-\frac {3 \theta} 2} \frac {\de \theta} {\sqrt {\pi \theta}}.% \\
%   &=&  \\
%   &=&
\end{eqnarray*}
Thanks to the inequality
$$4 \sqrt{\frac v {s\pi}} e^{-\frac {6 v} s}  \le \int_0^{\frac {4v} s} e^{-\frac {3 \theta} 2} \frac {\de \theta} {\sqrt {\pi \theta}} \le 4 \sqrt{\frac v {s\pi}},$$
we may write
\begin{eqnarray*}
 f(s,v)&=& (6\sqrt 2 . \pi^{-\frac32}c^{2k})s^{-\frac52}v^{\frac32+2k} e^{-2\frac{v^2}s + \frac v s K(s,v)},
\end{eqnarray*}
where $(s,v) \mapsto K(s,v)$ is continuous and bounded. The marginal
density of $\Tr_1$ is thus given by
\begin{eqnarray*}
\Inv {\de s} \Probup_1(\Tr_1 \in \de s) &=& \int_{\Real_+} f(s,v) \de v \\
 &=& (3\sqrt 2 . \pi^{-\frac32}c^{2k}) s^{-\frac 54 + k} \int_{\Real_+} w^{\Inv 4 + k} e^{-2 w + K(s,\sqrt {sw}) \sqrt {w/s}} \de w \\
 & \underset{s\to \infty} \sim& (3\sqrt 2 . \pi^{-\frac32}c^{2k}) s^{-\frac 54 + k} \int_{\Real_+} w^{\Inv 4 + k} e^{-2 w} \de w,
\end{eqnarray*}
where we used successively the change of variables $w=v^2/s$ and
dominated convergence theorem. Just integrate this equivalence in
the neighborhood of $+\infty$ to get
$$\Probup_1(\Tr_1 >t) \underset{t\to \infty}\sim c' t^{k-\Inv 4},$$
with the constant
 $$c'= \frac {3\sqrt 2 .
\pi^{-\frac32}c^{2k}} {\Inv 4 -k} \int_{\Real_+} w^{\Inv 4 + k}
e^{-2 w} \de w = \frac {3 c^{2k}} {\pi^{\frac 32} 2^{\frac 34+k}}
\cdot \frac {1+4k} {1-4k} \ \Gamma\left(\Inv 4 + k\right).$$
\end{proof}

%A second point of this proof differs largely from the one in \citejacobRLPI.
For now, we have introduced $\Pupz^*$, law of a process $(X,\dot X)$
indexed by $\Real_+^*$. We keep on following the proof of
\citejacobRLPI. First, we get that this law satisfies conditions
$(*)$ and $(**)$, and that for any $v>0$, the joint law of $\tau_v$
and $(X_{\tau_v+t}, \dot X_{\tau_v + t})_{t\ge 0}$ under $\Probup_u$
converges to that under $\Pupz^*$. Then we establish
Proposition~\ref{proposition_theoRLP} by controlling the behavior of
the process just after time 0, through the two following lemmas:
\begin{lemma}\label{convergence_X,X_0,0}
Under $\Pupz^*$, we have almost surely $(X_t,\dot X_t) \underset
{t\to 0} {\To} (0,0).$
\end{lemma}
This lemma allows in particular to extend $\Pupz^*$ to $\Real_+$. We
call $\Pupz$ this extension. The second lemma is more technical and
controls the behavior of the process on $[0,\tau_v[$ under
$\Probup_u$.

\begin{lemma}\label{lemma_majoration_M_u}
 Write $M_v= \sup\{|\dot X_t|, t \in [0,\tau_v[\}$. Then,
\begin{equation}\label{majoration_M_u}
\forall \eps>0, \forall \delta>0, \exists v_0>0, \exists u_0>0,
\forall 0<u\le u_0,\quad \Probup_u( M_{v_0} \ge \delta)\le \eps,
\end{equation}
\end{lemma}
In \citejacobRLPI, we proved these two results by using the
stochastic partial differential equation satisfied by the laws
$\Probc$. They are of course not available for the laws $\Probup$,
and we need a new proof. We start by showing a rather simple but
really useful inequality:
\begin{lemma} \label{P_x,1}
 The following inequality holds for any $(x,u) \in D$,
\begin{equation} \label{controle_Vr}
    \Probup_{x,u} \left(\Vr_1 / c \ge \frac {|u|} 2\right) \ge 1-\frac {\sqrt 3} \pi.
\end{equation}
\end{lemma}
For us, the important fact is that the probability is bounded below
by a positive constant, uniformly in $x$ and $u$. The constant
$1-\sqrt 3 / \pi$ is not intended to be the optimal one. Note that
this inequality will also be used again later on in this paper.

\begin{proof}[Proof of Lemma~\ref{P_x,1}]
For $u=0$, there is nothing to prove. By a scaling invariance
property we may suppose $u\in \{-1, 1\}$, what we do.

The density $f_{x,u}$ of $\Vr_1 /c$ under $\Prob_{x,u}^c$ is given
in Gor'kov \cite{Gorkov}. If you write $p_t(x,u;y,v)$ for the
transition densities of the (free) Kolmogorov process, given by
\begin{equation} \nonumber
p_t(x,u;y,v) = \frac{\sqrt 3}{\pi t^2} \exp \Big[ -\frac{6}{t^3}
(y-x-tu)^2 + \frac{6}{t^2} (y-x-tu)(v-u)-\frac{2}{t} (v-u)^2 \Big],
\end{equation}
and $\Phi(x,u;y,v)$ for its total occupation time densities, defined
by
$$ \Phi(x,u;y,v) := \int_0^\infty p_t(x,u;y,v) dt,$$
then the density $f_{x,u}$ is given by
\begin{eqnarray} %\label{h_vII}
f_{x,u}(v) &=& v \Big[ \Phi (x,u;0,-v) - \frac{3}{2\pi}
\int_0^\infty \frac{\mu^\frac{3}{2}}{\mu^3+1} \Phi(x,u;0,\mu v )
d\mu \Big]. \ \ \ \ \ \ \ \
\end{eqnarray}
Now, knowing the density of $\Vr_1$ under $\Prob_{x,u}^c$, we get that of $\Vr_1$ under $\Probup_{x,u}$ by multiplying it by the increasing function $v \mapsto v^{2k}$. This necessarily increases the probability of being greater than $c/2$. Consequently, it is enough to prove% the existence of a strictly positive constant $K'$ such that
$$\Prob_{x,u}^c (\Vr_1 / c \ge \frac 1 2) \ge K' $$
as soon as $u\in\{-1,1\}$. But very rough bounds give
\begin{eqnarray*}
f_{x,u}(v) &\le& v \Phi (x,u;0,-v) \\
&\le& v \int_0^\infty \frac{\sqrt 3}{\pi t^2} \exp(- \frac {(u+v)^2}
{2t}) \de t.
\end{eqnarray*}
For $u\in\{-1,1\}$ and $v \in [0,1/2]$ we have $|u+v|\ge 1/2$ and
thus
$$f_{x,u}(v) \le \frac{v \sqrt 3}{\pi} \int_0^\infty \Inv{t^2} \exp(-\Inv {8t}) \de t = \frac {8 \sqrt 3} {\pi} v. $$
Consequently,
$$\Prob_{x,u}(\Vr_1/c \ge \Inv 2) \ge 1- \int_0^{1/2} \frac {8 \sqrt 3} {\pi} v \de v = 1- \frac {\sqrt 3} \pi >0.$$
\end{proof}
\begin{proof}[Proof of Lemma~\ref{convergence_X,X_0,0}]
%Now, let us prove $(X_t, \dot X_t) \to (0,0)$ a.s.~under $\Pupz^*$. %The construction of the law $\Pupz^*$ implies
First, observe that conditions $(*)$ and $(**)$ imply that the
variables $\tau_v=\inf\{t>0, X_t=0, \dot X_t > v\}$ and $\tau_v^-:=
\sup\{t<\tau_v, X_t=0\}$ are almost surely strictly positive and go
to zero when $v$ goes to zero. Then, observe that is is enough to
show the almost sure convergence of $\dot X_t$ to 0 when $t \to 0$,
 and suppose on the
contrary that this does not hold.

 Then, there would exist a positive $x$ such that
$\Pupz^*(T_x=0)>0$, where we have written $T_x:=\inf\{t>0, |\dot
X_t|>x\}.$ By self-similarity this would be true for any $x>0$ and
in particular we would have
\begin{equation} \label{T_1=0_probaK}
K:= \Pupz^*(T_1=0)>0.
\end{equation}
%The idea is to prove that the variable $\tau_{c/2}^-$ would then take the value 0 with positive probability, providing us the desired contradiction.
Informally, this, together with \eqref{controle_Vr}, should induce
that $\tau_{c/2}^-$ takes the value zero with probability at least
$(1 - \sqrt 3 / \pi) K$, and give the desired contradiction. However
it is not straightforward, because we cannot use a Markov property
%\footnote{At this point we only know that the process $\{(X_t,\dot X_t), t>0\} is Markovian, and we certainly can't just include time 0 without justification.}
at time $T_1$, which can take value 0, while the process is still
not defined at time 0. Consider the stopping time $T_1^\eps:=\inf
\{t>\eps, |\dot X_t|>x\}$. For any $\eta>0$, we have
$$\liminf_{\eps \to 0} \Pupz^*(T_1^\eps < \eta) \ge \Pupz^*(\liminf_{\eps \to 0} \{T_1^\eps < \eta\} ) \ge \Pupz^*(T_1<\eta) \ge K,$$
and in particular there is some $\eps_0(\eta)$ such that for any
$\eps<\eps_0(\eta)$,
\begin{equation} \label{T_1^eps}
 \Pupz^*(T_1^\eps < \eta) \ge \frac K 2.
\end{equation}
Now, write $\theta$ for the translation operator defined by
$\theta_x((X_t)_{t\ge0}) = (X_{x+t})_{t\ge0}$, so that $\Vr_1 \circ
\theta_{T_1^\eps}$ denotes the velocity of the process at its first
bounce after time $T_1^\eps$. From \eqref{T_1^eps} and
Lemma~\ref{P_x,1}, a Markov property gives, for $\eps <
\eps_0(\eta)$,
$$\Pupz^* \left( T_1^\eps < \eta, \Vr_1 \circ \theta_{T_1^\eps} \ge \frac c 2 \right) \ge K':= \Bigg(1-\frac {\sqrt 3} {\pi}\Bigg) \frac K 2.$$
We have \emph{a fortiori} $\Pupz^*(\tau_{c/2}^- \le \eta) \ge K'.$
This result true for any $\eta>0$ leads to $\Pupz^*(\tau_{c/2}^- =
0) \ge  K' >0$, and we get a contradiction. This shows $(X_t,\dot
X_t) \underset{t\to 0} \To (0,0)$ under $\Pupz^*$, as requested.
\end{proof}
\begin{proof}[Proof of Lemma~\ref{lemma_majoration_M_u}]
We should prove~\eqref{majoration_M_u}. Fix $\eps, \delta>0$. The
event $\{M_v\ge \delta\}$ coincides with the event $T_\delta \le
\tau_v$. From a Markov property at time $T_\delta$
and~\eqref{controle_Vr}, we get, for any $v<c \delta / 2$, and any
$u$,
$$(1-\sqrt 3/\pi)  \Probup_u(M_v \ge \delta) \le \Probup_u(\dot X_{\tau_v} \ge c \delta / 2).$$
Choose $v_0$ such that $\Pupz(\dot X_{\tau_{v_0}} \ge c \delta / 2)
\le \eps$. Then, from the convergence of the law of $\dot
X_{\tau_{v_0}}$ under $\Probup_u$ to that under $\Pupz$, we get, for
$u$ small enough,
$$\Probup_u(\dot X_{\tau_{v_0}} \ge c \delta / 2) \le 2 \eps,$$
and hence
$$ \Probup_v(M_{v_0} \ge \delta) \le \frac {2 } {1 - \sqrt 3 / \pi}\ \eps.$$
\end{proof}
In conclusion, all this suffices to show
Proposition~\ref{proposition_theoRLP}.

\section{The resurrected process} \label{section_Resurrect}

\subsection{It\=o excursion measure, recurrent extensions,\\
 and $(SOR)$ equations}
%\subsection{Recurrent extensions and excursion measures}

We finally tackle the problem of interest, that is the recurrent
extensions of the reflected Kolmogorov process. A recurrent
extension of the latter is a Markov process that behaves like the
reflected Kolmogorov process until $\Tri$, the hitting
time %\footnote{In this section there may be (and actually there will be) an infinite number of bounces just after the initial time, so that $\Tri$, the hitting time of $(0,0)$ is no more equal to  $\sup \Tr_n$.}
 of $(0,0)$, but that is defined for any positive times and
does not stay at $(0,0)$, in the sense that the Lebesgue measure of
the set of times when the process is at $(0,0)$ is almost surely 0.
More concisely, we will call such a process a resurrected reflected
process.

We recall that It\=o's program and results of Blumenthal
\cite{Blumenthal83} establish an equivalence between the law of
recurrent extensions  of a Markov process and excursion measures
compatible with its semigroup, here $\Psg_t^c$ (where as usually in
It\=o's excursion theory we identify the measures which are equal up
to a multiplicative constant). The \emph{set of excursions} $\E$ is
defined by
$$ \E := \{(x,\dot x) \in \Cbb| \Tri>0 \text{ and }x_t \un_{t\ge \Tri} = 0 \}.$$
An excursion measure $n$ compatible with the semigroup $\Psg_t^c$ is
defined by the three following properties:
\begin{enumerate}
\item The measure $n$ is carried by $\E$.
\item For any $\F_\infty-$measurable function $F$ and any $t>0$, any $A \in \F_t$,
 $$ n(F\circ \theta_t,A \cap\{t<\Tri\}) = n( \Prob_{X_t,\dot X_t}^c(F),A \cap \{t<\Tri\}).$$
\item $ n(1-e^{-\Tri})<\infty.$
\end{enumerate}
%Up to a multiplicative constant, the excursion measure of a Markov process is unique
We also say that $n$ is a pseudo-excursion measure compatible with
the semigroup $\Psg_t^c$ if only the two first properties are
satisfied and not necessarily the third one. We recall that the
third property is the necessary condition in It\=o's program in
order for the lengths of the excursions to be summable, hence in
order for It\=o's program to succeed.
%Finally, giving an excursion measure $n$ is equivalent to giving the
%entrance law of this measure, defined by
% $$n_s( \de x,\de u) := n((X_s,\dot X_s)\in \de x \otimes \de u, s < \Tri)$$
%for $s>0.$
 Besides, we are here interested in recurrent extensions which leave $(0,0)$ continuously. % for obvious physical reasons
These extensions correspond to excursion measures $n$ which satisfy
the additional condition $n((X_0,\dot X_0) \ne (0,0))=0$.
%Finally,
%these measures correspond to entrance laws $n_s$ which satisfy the
%additional condition $\lim_{s\to0} n_s(D^0 \setminus B) =0 $ for any
%$B$ neighborhood of $(0,0)$.
 Our main results are the following:

\begin{theorem} \label{Theo_mesure}
There exists, up to a multiplicative constant, a unique excursion
measure $\n$ compatible with the semigroup $\Psg_t^c$ and such that
$\n((X_0,\dot X_0) \ne (0,0) )=0$. We may choose $\n$ such that
\begin{equation} \label{queue_Tri_sous_n}
\n(\Tri>s) = C_1 s^{-k},
\end{equation}
where $C_1$ is the constant defined by~\eqref{valeur_C_1}, and
$k=k(c)$ has been introduced in Lemma~\ref{lemme_queueTri}. The
measure $\n$ is then characterized by any of the two following
formulas:
\begin{eqnarray} \label{formula1}
\n(f(X,\dot X), \Tri>T) &=& \Pupz(f(X,\dot X) H(X_T,\dot X_T) ^{-1}
),
\end{eqnarray}
for any $\F_t-$stopping time $T$ and any $f$ positive measurable
functional depending only on $(X_t, \dot X_t)_{0\le t \le T}$.
\begin{eqnarray}
\label{formula2} \n (f(X,\dot X), \Tri>T) &=& \lim_{(x,u) \to (0,0)}
H(x,u)^{-1} \Prob_{x,u}^c(f(X,\dot X), \Tri>T),
\end{eqnarray}
for any $\F_t-$stopping time $T$ and any $f$ positive
\emph{continuous} functional depending only on $(X_t, \dot
X_t)_{0\le t \le T}$.
\end{theorem}
So It\=o's program constructs a Markov process with associated It\=o excursion measure $\n$ and that spends no time at $(0,0)$, %(the inverse local time is subordinator with drift 0),
 that is a recurrent extension, that is a resurrected reflected process. We call its law $\Pcz$.
The second theorem will be the weak existence and solution to
equations $(SOR)$, the law of any solution being given by $\Pcz$. It
is implicit in this theorem and until the end of the paper that the
initial condition is $(0,0)$, though this generalizes easily to any
other initial condition $(x,u) \in D$.

\begin{theorem} \label{theo_SORL}
The law $\Pcz$ gives the unique solution, in the weak sense, of
equations $(SOR)$:

 $\bullet$ Consider $(X, \dot X)$ a process of law $\Pcz$. Then the jumps of $\dot X$ on any finite interval are summable and the process $W$ defined by
$$W_t=\dot X_t + (1+c) \sum_{0<s\le t} \dot X_{s-} \un_{X_s=0}$$
is a Brownian motion. As a consequence the triplet $(X,\dot X,W)$ is
a solution to $(SOR)$.

 $\bullet$ For any solution $(X,\dot X, W)$ to $(SOR)$, the law of $(X, \dot X)$ is $\Pcz$.
\end{theorem}

Before we tackle the proof these theorems, let us write some
comments and consequences. First, the It\=o excursion measure $\n$
is entirely determined by its entrance law, which is defined by
$$\n_s( \de x,\de u) := n((X_s,\dot X_s)\in \de x \otimes \de u, s < \Tri)$$
for $s>0.$ But Theorem~\ref{Theo_mesure} implies that it is
characterized by any of the two following formulas:
\begin{eqnarray} \label{formula1_entrance}
\n_s (f) &=& \Pupz(f(X_s,\dot X_s) H(X_s,\dot X_s) ^{-1} ), \quad
s>0,
\end{eqnarray}
for $f:D^0 \to \Real_+$ measurable.
\begin{eqnarray}
\label{formula2_entrance} \n_s (f) &=& \lim_{(x,u) \to (0,0)}
H(x,u)^{-1} \Prob_{x,u}^c(f(X_s,\dot X_s),\Tri>s), \quad s>0,
\end{eqnarray}
for $f:D^0 \to \Real_+$ continuous.

Formulas similar to these are found in the case of self-similar
Markov processes studied by Rivero \cite{rivero04}. This ends the
parallel between our works. Rivero underlined that the self-similar
Markov process conditioned on never hitting 0 that he introduced
plays the same role as the Bessel process for the Brownian motion.
In our model, this role is played by the reflected Kolmogorov
process conditioned on never hitting $(0,0)$. Here is a short
presentation of this parallel. Write $P_x$ for the law of a Brownian
motion starting from position $x$, $\widetilde P_x$ for the law of
the ``three-dimensional'' Bessel process starting from $x$. Write
$n$ for the It\=o excursion measure of the absolute value of the
Brownian motion (that is, the Brownian motion reflected at 0), and
$\Tr$ for the hitting time of 0. Then the inverse function is
excessive (i.e nonnegative and superharmonic) for the Bessel process
and we have the two well-known formulas
\begin{eqnarray*}
\n(f(X), \Tr>T) &=& \widetilde P_0 (f(X)/X_T ) \\
\n (f(X), \Tr>T) &=& \lim_{x \to 0} \Inv x P_x (f(X), \Tr>T),
\end{eqnarray*}
for any $\F_t-$stopping time $T$ and any $f$ positive measurable functional (resp. continuous functional for the second formula)  depending only on $(X_t)_{0\le t \le T}$. %The similarity with our results is clear.% continuous functional for the second formula
\espace

Now, let us give an application of Formula \eqref{queue_Tri_sous_n}.
%Consider $(X,\dot X)$ for the recurrent extension that has $\n$ for It\=o's excursion measure (see It\=o's reconstruction of a Markov process from its excursion measure, \cite{}).
Write $l$ for the local time spent by $X$ at zero, under $\Pcz$.
Formula  \eqref{queue_Tri_sous_n} implies that the inverse local
time $l^{-1}$ is a subordinator with jumping measure $\Pi$
satisfying $\Pi(\Tri>s) \propto s^{-k}.$ That is, it is a stable
subordinator of index $k$. A well-known result of Taylor and Wendel
\cite{TaylorWendel} then gives that the exact Hausdorff function of
the closure of its range (the range is the image of $\Real_+$ by
$l^{-1}$) is given by $\phi(\eps)=\eps^{k} (\ln \ln 1/\eps)^{1-k}$
almost surely. The closure of the range of $l^{-1}$ being equal to
the zero set $\mathcal Z :=\{t \ge 0: X_t=\dot X_t=0\}$, we get the
following corollary:
\begin{corollary}\label{corollary}
The  exact Hausdorff function of the set of the passage times to
$(0,0)$ of the resurrected reflected Kolmogorov process is
$\phi(\eps)=\eps^{k} (\ln \ln 1/\eps)^{1-k}$ almost surely.
\end{corollary}
It is also clear that the set of the bouncing times of the
resurrected reflected Langevin process -- the moments when the
process is at zero with a nonzero speed -- is countable. Therefore
the zero set of the resurrected reflected Langevin process has the
same exact Hausdorff function. \espace

Finally, we should mention that the self-similarity property enjoyed
by the Kolmogorov process easily spreads to all the processes we
introduced. If $a$ is a positive constant, denote by $(X^a, \dot
X^a)$ the process $(a^3 X_{a^{-2}t}, a X_{a^{-2}t})_{t\ge 0}$. Then
the law of $(X^a, \dot X^a)$ under $\Probc_{x,u}$ is simply
$\Probc_{a^3x,au}$. We have $H(a^3x, au)=a^{2k}H(x,u).$ The law of
$(X^a, \dot X^a)$ under $\Probup_{x,u}$, resp. $\Pupz$, is simply
$\Probup_{a^3x,au}$, resp. $\Pupz$. Finally, the measure of $(X^a,
\dot X^a)$ under $\n$ is simply $a^{2k}\n$.

\espace Last two subsections are devoted to the proof of the two
theorems.

\subsection{The unique recurrent extension compatible with $\Psg_t^c$}

\subsubsection*{Construction of the excursion measure}

%We prove the existence in a constructive way.
The function $1/H$ is excessive for the semigroup $\widetilde
\Psg_t$ and the corresponding $h-$transform is $\Psg_t^c$ (see
Definition~\ref{def_Pup}). Write $\n$ for the $h-$tranform of
$\Pupz$ via this excessive function $1/H$. That is, $\n$ is the
unique measure on $\Ccal$ carried by $\{\Tri>0\}$ such that under
$\n$ the coordinate process is Markovian with semigroup $\Psg_t^c$,
and for any $\F_t-$stopping time $T$ and any $A_T$ in $\F_T$, we
have
$$\n(A_T, T<\Tri) = \Pupz(A_T, H(X_T,\dot X_T)^{-1}).$$

Then, $\n$ is a pseudo-excursion measure compatible with semigroup
$\Psg_t^c$, which verifies $\n((X_0,\dot X_0) \ne (0,0)) = 0$ and
satisfies Formula~\eqref{formula1}. For $f$ continuous functional
depending only on $(X_t,\dot X_t)_{t\le T}$, we have
\begin{eqnarray*}
\Pupz(f(X_s,\dot X_s) H(X_s,\dot X_s)^{-1} )
 &=& \lim_{(x,u) \to (0,0)} \Probup_{x,u}(f(X_s,\dot X_s) H(X_s,\dot X_s)^{-1}) \\
 &=& \lim_{(x,u) \to (0,0)} \Inv {H(x,u)} \Prob_{x,u}^c (f(X_s,\dot X_s), \Tri>s),
\end{eqnarray*}
so that the pseudo-excursion measure $\n$ also satisfies
Formula~\eqref{formula2}. In particular, taking $T=s$ and $f=1$, and
considering the limit along the half-line $x=0$, this gives
$$ \n(\Tri>s) = \lim_{u\to 0} u^{-2k} \Prob_{0,u}(\Tri>s).$$
Using Lemma~\ref{lemme_queuebis} and the scaling invariance
property, we get
 \begin{equation} %\label{queue_Tri_sous_n}
\nonumber \n(\Tri>s) = C_1 s^{-k},
\end{equation}
where $C_1$ is the constant defined by~\eqref{valeur_C_1}. This is exactly Formula~\eqref{queue_Tri_sous_n}. %Finally we should verify that $\n$ is an excursion measure. But Formula~\eqref{queue_Tri_sous_n}
This formula gives, in particular, $$ \n(1-e^{-\Tri}) = C_1
\Gamma(1-k),$$ where $\Gamma$ denotes the usual Gamma function.
Hence, $\n$ is an excursion measure.

\espace

Finally, in order to establish Theorem~\ref{Theo_mesure} we just
should prove that $\n$ is the only excursion measure compatible with
the semigroup $\Psg_t^c$ such that $\n((X_0,\dot X_0) \ne (0,0)
)=0$. That is, we should show the uniqueness of the law of the
resurrected reflected process.
%\end{proof}

\subsubsection*{Uniqueness of the excursion measure}

Let $\n'$ be such an excursion measure, compatible with the
semigroup $\Psg_t^c$, and satisfying $\n'((X_0,\dot X_0) \ne (0,0)
)=0$. We will prove that $\n$ and $\n'$ coincide, up to a
multiplicative constant. Recall that $\Tr_1$ is defined as the
infimum of $\{t>0,X_t=0\}$.
\begin{lemma} \label{Tr_1=0} The measure $\n'$ satisfies:  $$ \n'(\Tr_1\ne 0)=0$$
\end{lemma}
\begin{proof}
This condition will appear to be necessary to have the third
property of excursion measures, that is $ \n'(1-e^{-\Tri})<\infty.$
Suppose on the contrary that $\n'(\Tr_1\ne 0)>0$ and write $\tilde
\n(\cdot)=\n'(\cdot \un_{\Tr_1\ne 0})$. The measure $\tilde \n$ is
an excursion measure compatible with the semigroup $\Psg_t^c$ such
that $\tilde \n((X_0,\dot X_0) \ne (0,0) )=0$, satisfying $\tilde
\n(\Tr_1=0)=0.$ Consider $\nbis((X_t,\dot X_t)_{t\ge 0}) := \tilde
\n((X_t \un_{t<\Tr_1},\dot X_t \un_{t<\Tr_1})_{t\ge0})$ the
excursion measure of the process killed at time $\Tr_1$.

The measure $\nbis$ is an excursion measure compatible with the
semigroup $\Psg_t^0$, semigroup of the Kolmogorov process killed at
time $\Tr_1$ (the first hitting time of $\{0\} \times \Real$).
Therefore its first marginal must be the excursion measure of the
Langevin process reflected on an inelastic boundary, introduced and
studied in \cite{reflecting}. In particular, under $\nbis$, the
absolute value of the incoming speed at time $\Tr_1$, or $|\dot
X_{\Tr_1 -}|$, is distributed proportionally to $v^{-\frac 32} \de
v$ (see \cite{reflecting}, Corollary 2, (ii)). This stays true under
$\nti$ and implies that $\Vr_1=c |\dot X_{\Tr_1 -}|$ is also
distributed proportionally to $v^{-\frac 32} \de v$. Now, a Markov
property at the stopping time $\Tr_1$ under $\nti$ gives
$$ \nti(\Tri-\Tr_1>t|\Vr_1=v) = \Prob_v^c(\Tri>t) = \Prob_1^c(\Tri>v^{-2}t)  \underset{v^{-2}t \to \infty}\sim  C v^{2k} t^{-k}$$
As a consequence the function $v \mapsto v^{-\frac 32}
\nti(\Tri-\Tr_1>t|\Vr_1=v)$ is not integrable in the neighborhood of
0. That is $\nti(\Tri-\Tr_1>t) = +\infty$, we get a contradiction.
\end{proof}

Recall that we owe to prove that $\n'$ and $\n$ are equal, up to a
multiplicative constant. Let us work on the corresponding entrance
laws. Take $s>0$ and $f$ a bounded continuous function. It is
sufficient to prove $\n'_s(f) = C \n_s(f)$, where $C$ is a constant
independent of $s$ and $f$.

By reformulating Lemma~\ref{Tr_1=0}, time $\Tr_1$ is zero
$\n'$-almost surely, in the sense that the $\n'$-measure of the
complementary event is 0. That is, $\n'$-a.s., the first coordinate
of the process comes back to zero just after the initial time, while
the second coordinate cannot be zero, for the simple reason that we
are working on an excursion outside from $(0,0)$. This, together
with the fact that the velocity starts from $\dot X_0=0$ and is
right-continuous, implies that $\n'$-almost surely, the time
$\tau_v$ (which, we recall, is the instant of the first bounce with
speed greater than $v$) is going to $0$ when $v$ is going to 0.

%\noindent
 We deduce, by dominated convergence, from the continuity of $f$, and, again, from the
right-continuity of the paths, that
\begin{equation} \label{n'_s=lim}
\n'_s(f) = \lim_{u\to 0} \n'(f(X_{s+\tau_v},\dot X_{s+\tau_v})
\un_{\tau_v<\infty, \Tri>s+\tau_v}).
\end{equation}
An application of the Markov property gives
\begin{eqnarray*}
  \n'(f(X_{s+\tau_v},\dot X_{s+\tau_v}) \un_{\tau_v<\infty, \Tri>s+\tau_v}) &=&
  \int_{\Real_+} \n'(\dot X_{\tau_v} \in \de u) \Probc_u(f(X_s, \dot X_s)\un_{\Tri>s}) \\
   &=& \int_{\Real_+} \n'(\dot X_{\tau_v} \in \de u) u^{2k}g(u),
\end{eqnarray*}
where $g(u) = u^{-2k} \Probc_u(f(X_s, \dot X_s)\un_{\Tri>s})=
H(0,u)^{-1} \Probc_u(f(X_s, \dot X_s)\un_{\Tri>s})$ converges to
$\n_s(f)$ when $u\to 0$, by Formula~\eqref{formula2}. Moreover the
function $u^{2k} g(u)$ is bounded by $\|f\|_\infty$, and for any
$\eps>0$ we have $\n'(\dot X_{\tau_v}>\eps) \to 0$ when $v \to 0$.
Informally, all this explains that when $v$ is small, all the mass
in the integral is concentrated in the neighborhood of $0$, where we
can replace $g(u)$ by $\n_s(f)$. More precisely, write
$$ \int_{\Real_+} \n'(\dot X_{\tau_v} \in \de u) u^{2k}g(u) = I(v) + J(v),$$
where
\begin{eqnarray*}
 I(v) &=& \int_0^1 \n'(\dot X_{\tau_v} \in \de u) u^{2k} \n_s(f), \\
 &&\\
 J(v) &=& \int_0^\infty \n'(\dot X_{\tau_v} \in \de u) u^{2k} (g(u) - \n_s(f) \un_{u\le 1}).
\end{eqnarray*}
By splitting the integral defining $J(v)$, we deduce that $J(v)$ is
negligible compared to $1 \vee I(v).$ Recalling that the sum $ I(v)
+ J(v)$ converges to $\n'_s(f)$ (Formula \eqref{n'_s=lim}), we get
that $I(v)$ converges to $\n'_s(f)$ when $v\to 0$, while $J(v)$
converges to 0.

We thus have
$$\n'_s(f)= C \n_s(f),$$
where $C$ is independent of $s$ and $f$ and given by
$$ C= \lim_{v\to 0} \int_0^1 \n'(\dot X_{\tau_v} \in \de u) u^{2k}.$$
Uniqueness follows. Theorem~\ref{Theo_mesure} is proved.

\subsection{The weak unique solution to the $(SOR)$ equations}

We now prove Theorem~\ref{theo_SORL}.

\subsubsection*{Weak solution}

We consider, under $\Pcz$, the coordinate process $(X,\dot X)$, and
its natural filtration $(\F_t)_{t\ge0}$. We first prove that the
jumps of $\dot X$ are almost-surely summable on any finite interval.
As there are (a.s.) only finitely many jumps of amplitude greater
than a given constant on any finite interval, it is enough to prove
that the jumps of amplitude less than a given constant are (a.s.)
summable. Write $L$ for a local time of the process $(X,\dot X)$ in
$(0,0)$, $L^{-1}$ its inverse, and $\n$ the associated excursion
measure. It is sufficient to prove that the expectation of the sum
of the jumps of amplitude less than $1+1/c$ (jumps at the bouncing
times for which the outgoing velocity is less than one), and
occurring before time $L^{-1}(1)$, is finite. This expectation is
equal to
$$(1+\Inv c) \int_0^1 \n(N_{[v,1]}(X,\dot X)) \de v,$$
where we write $N_{I}(X,\dot X)$ for the number of bounces of the
process $(X,\dot X)$ with outgoing speed included in the interval
$I$. For a fixed  $v$, introduce the sequence of stopping times
defined by $\tau^v_0=0$ and $\tau^v_{n+1} = \inf \{t>\tau^v_n, X_t
=0, \dot X_t \in [v,1]\}$ for $n\ge0$. Then $N_{[v,1]}(X,\dot X)$ is
also equal to $\sup \{n, \tau^v_n < \Tri\}$. Thanks to formula
\eqref{formula1}, for any $n>0$, we have:
\begin{eqnarray*}
 \n(\Tri > \tau^v_n )&=& \Pupz(H(X_{\tau^v_n},\dot X_{\tau^v_n})^{-1} \un_{\tau^v_n<\infty} ) \\
   &=& \Pupz(\dot X_{\tau^v_n}^{-2k} \un_{\tau^v_n<\infty} ) \\
   &\le& v^{-2k} \Pupz( \tau^v_n<\infty).
\end{eqnarray*}
As a consequence, we have
\begin{eqnarray*} \n(N_{[v,1]}(X,\dot X))&\le& v^{-2k} \Pupz(\sup \{n, \tau^v_n < \Tri\})\\
&\le&  v^{-2k} \Proup(N_{[\ln v,0]}^d(S)),
\end{eqnarray*}
where we have written $N_{[\ln v,0]}^d(S)$ for the number of
instants $n\in\Z$ such that $S_n \in [\ln v,0]$. Recall also that
$\Proup$ is the law of the spatially stationary random walk. It is
now a simple verification that $\Proup(N_{[\ln v,0]}^d(S))$ is
finite and proportional to the length of the interval $[\ln(v),0]$,
that is $-\ln v $. It follows
$$\n(N_{[v,1]}(X,\dot X)) \underset {v\to 0}= O(v^{-2k} \ln(1/v)) $$
and (recall $k<1/4$)
$$ (1+\Inv c) \int_0^1 \n(N_{[v,1]}(X,\dot X)) \de v < \infty.$$
The jumps are summable.

\espace

Now, write
$$W_t= \dot X_t + (1+c) \sum_{0<s \le t} \dot X_{s-} \un_{X_s=0}.$$
We aim to show that the continuous process $W$ is a Brownian motion.
For $\eps>0$, we introduce the sequence of stopping times
$(T_n^{\eps})_{n\ge 0}$ defined by $T_0^\eps = 0$ and, for $n\ge 0$,
\begin{equation}
\left\{ \nonumber \begin{array}{rcl}
 T_{2n+1}^\eps &=& \inf\{t>T_{2n}^\eps, X_t=0, \dot X_t>\eps\} \\
T_{2n+2}^\eps &=& \inf\{t>T_{2n+1}^\eps, X_t=\dot X_t=0\}
\end{array}\right.
\end{equation}
We also introduce $F^\eps= \bigcup_{n\ge 0} [T_{2n}^\eps,
T_{2n+1}^\eps]$ and $H_t^\eps= \un_{F^\eps} (t)$. For
$0<\eps'<\eps$, we have $H^{\eps'}\le H^\eps$, or equivalently,
$F^{\eps'} \subset F^\eps$. When $\eps$ goes to $0+$, $F^\eps$
converges to the zero set $\mathcal Z =\{t, X_t=\dot X_t=0\}$, and
$H^\eps$ converges pointwisely to $H^0= \un_{\mathcal Z}$. Note that
the processes $H^\eps$ and $H^0$ are $\F_t-$adapted. Note, also,
that Corollary~\ref{corollary} implies in particular that $\mathcal
Z$ has zero Lebesgue measure. For ease of notations, we will
sometimes omit the superscript $\eps$.

Conditionally on $\dot X_{T_{2n+1}} =u$, the process
$(X_{(T_{2n+1}+t)\wedge T_{2n+2}})_{t\ge 0}$ is independent of
$\F_{T_{2n+1}}$ and has law $\Prob_u^c$. As a consequence the
process $(W_{(T_{2n+1}+t)\wedge T_{2n+2}} - W_{T_{2n+1}})_{t\ge 0}$
is a Brownian motion stopped at time $T_{2n+2}- T_{2n+1}$.  Write
$$ W_t= \int_0^t H_s^\eps \de W_s + \int_0^t (1- H_s^\eps) \de W_s.$$
The process $\int_0^t (1 - H_s^\eps) \de W_s$ converges almost
surely to $\int_0^t (1 - H_s^0) \de W_s$. But the process $\int_0^t
(1- H_s^0) \de W_s$ is a continuous martingale of quadratic
variation $\int_0^t (1 - H_s^0) \de s = t$ and thus a Brownian
motion. In order to prove that it actually coincides  with $W$, we
just need to prove that the term $D_t^\eps := \int_0^t H_s^\eps \de
W_s$ is almost-surely converging to $0$ when $\eps \to 0$. Without
loss of generality, we just prove it on the event $t\le L^{-1}(1)$.

This term can be rewritten as
$$ D_t^\eps= \left\{ \begin{array}{ll}
\displaystyle \sum_{k\le n} \big( W_{T_{2k+1}} - W_{T_{2k}}\big) & \text{if }T_{2n+1} \le t <T_{2n+2},  \\
 \\
\displaystyle W_t-W_{T_{2n}} + \sum_{k< n} \big( W_{T_{2k+1}} - W_{T_{2k}}\big) \quad & \text{if } T_{2n} \le t<T_{2n+1}.
\end{array}
\right. $$
 Now, for any $k$, we have
$$ W_{T_{2k+1}} - W_{T_{2k}} = \dot X_{T_{2k+1}} + (1+c) \sum_{T_{2k}<s \le T_{2k+1}} \dot X_{s-} \un_{X_s=0},$$
and for any $T_{2n} \le t<T_{2n+1}$,
$$ W_t - W_{T_{2n}} = \dot X_t + (1+c) \sum_{T_{2n}<s \le t} \dot X_{s-} \un_{X_s=0},$$
Hence the term $D_t^\eps$ involves jumps of amplitude less than
$(1+c) \eps$, whose sum is going to 0 when $\eps$ goes to zero, plus
the fraction $c/(1+c)$ of the jumps occurring at times $T_{2k+1}$,
plus the possible extra term $\dot X_t$, not corresponding to any
jump. We will prove nonetheless that the jumps occurring at times
$T_{2k+1}$, and $|\dot X_t|$, are all small when $\eps$ is small
enough. It will follow that $D_t^\eps$ tends to 0 when $\eps$ goes
to 0.

Fix $\eta>0$. Write $\Ae$ for the event $$\sup_{s \le L^{-1}(1), s
\in F^\eps} \dot X_s \ge \eta.$$ We will prove that the probability
of $\Ae$ is going to 0 when $\eps$ goes to 0, so that we almost
surely don't lie in $\Ae$ for $\eps$ small enough, and as a
consequence the jumps occurring at times $T_{2k+1}$ and the possible
term $|\dot X_t|$ will then all be less than $\eta$, as requested.
Write $\widetilde T^\eps$ for the infimum of $\{t: t \in F^\eps,
|\dot X_t| \ge \eta\}$ and $n_\eps$ for the supremum of $\{n, T_{2n}
\le \widetilde T^\eps\}$. The event $\Ae$ coincides with
$\{\widetilde T^\eps<L^{-1}(1)\}$ or $\{T_{2 n_\eps+1}
<L^{-1}(1)\}$.

The Markov property at the stopping time $\widetilde T^\eps$, together with the
inequality~\eqref{controle_Vr}, gives
$$ \Prob(\{\dot X_{T_{2n_\eps +1}} \ge \eta c /2\} \cap \Ae) \ge \big( 1-{\sqrt 3} / \pi \big)  \Prob(\Ae).$$
The event $\{\dot X_{T_{2n_\eps +1}} \ge \eta c / 2\} \cap \Ae$
is contained in the event that there is an excursion occurring before time
$L^{-1}(1)$ for which the first bounce with speed greater than
$\eps$ is actually greater than $\eta c/2$. This event has
probability
$$ \n(T_1^\eps < \infty, \dot X_{T_1^\eps} \ge \eta c /2),$$
where $T_1^\eps$ is still defined as the time of the first bounce
with speed greater than $\eps$, here for the excursion. We have:
\begin{eqnarray*}
  \n( \dot X_{T_1^\eps} \ge \eta c /2, \Tri>T_1^\eps) &=& \Pupz( H(0,\dot X_{T_1^\eps})^{-1} \un_{\dot X_{T_1^\eps} \ge \eta c /2})  \\
   &\le& (\eta c /2)^{-2k} \Pupz(\dot X_{T_1^\eps} \ge \eta c /2) \\
   &\le& (\eta c /2)^{-2k} \overup\big( ]\ln( {\eta c}/({2\eps})), \infty[ \big),
\end{eqnarray*}
where we recall that $\overup$ is the stationary law of the
overshoot appearing in Proposition \ref{proposition_theoRLP}. This
probability is thus going to 0 when $\eps$ goes to 0, as well as
$\Prob(\Ae)$.

The process $W$ is a Brownian motion, and $(X,\dot X, W)$ is a
solution to Equations $(SOR)$.

\espace

\subsubsection*{Weak uniqueness}

Consider $(X,\dot X, W)$, with law $\Prob$, be any solution to $(SOR)$, and its
associated filtration $(\F_t)_{t\ge0}$. Then we have
$$ \dot X_t =W_t - (1+c) \sum_{0<s \le t} \dot X_{s-} \un_{X_s=0},$$
with $W$ a Brownian motion.

We start with the observation that the process $\dot X$ does not
explode and that the sum just involves positive jumps. Therefore
these jumps are summable.
 But the process
$\sum_{0<s \le t} \dot X_{s-} \un_{X_s=0}$ is adapted, hence $\dot
X$ is a semimartingale. As a consequence, it possesses local times
$(L^a)_{a \in \Real}$, and we have an occupation formula (see for
example \cite{Protter}, Theorem~70 Corollary~1, p216):
$$ \int_{-\infty}^{+\infty} L^a_t g(a) \de a = \int_0^t g( \dot X_{s-}) \de s,$$
for any $g$ bounded measurable function. Taking $g= \un_{\{0\}}$ shows that $\dot X$ spends no time at zero. %the indicator function of 0.
It follows that the process $(X,\dot X)$ spends no time at $(0,0)$.

\espace

Now, exactly as before, introduce, for $\eps>0$, the sequence of
stopping times $T_n^\eps$, defined by $T_0^\eps=0$ and
\begin{equation}
\left\{ \nonumber \begin{array}{rcl}
 T_{2n+1}^\eps &=& \inf\{t>T_{2n}^\eps, X_t=0, \dot X_t>\eps\} \\
T_{2n+2}^\eps &=& \inf\{t>T_{2n+1}^\eps, X_t=\dot X_t=0\},
\end{array}\right.
\end{equation}
as well as $F^\eps= \bigcup_{n\ge 0} [T_{2n}^\eps, T_{2n+1}^\eps]$
and $H^\eps= \un_{F^\eps}$. Finally, define the closed set
$F=\lim_{\eps \to 0} F^\eps$ and the adapted process $H^0=\un_{F}.$
\begin{lemma}\label{lemma_F_mesure_nulle}
The set $F$ has almost
surely zero Lebesgue measure.
\end{lemma}
This result is not immediate. First, observe that the excursions of
the process may be of two types. Either an excursion bounces on the
boundary just after the initial time, or it doesn't. We call $\E_1$
the set of excursions of the first type, defined by
$$ \E_1:= \{(x,\dot x) \in \E| \Tr_1(x,\dot x) := \inf\{t>0, x_t=0\} = 0 \},$$
and $\E_2= \E \backslash \E_1$ the set of excursions of the second
type. Unlike before, we do not know \emph{a priori} that all the
excursions of the process lie in $\E_1$. If the process starts an
excursion at time $t$, we write $e^t$ for the corresponding
excursion.

A close look at $F$ shows that it contains not only the zero set
$\mathcal Z$, but also all the intervals $[t,t+\Tr_1(e^t)]$, where
$t$ is the starting time of an excursion $e^t \in \E_2$. Prove
Lemma~\ref{lemma_F_mesure_nulle} is equivalent to prove that there
is actually no excursion in $\E_2$.

Suppose that this fails. Then the process
$$ L(t)= \int_0^t H^0_s \de s $$
is not almost surely constantly equal to zero. We introduce its
right-continuous inverse
$$L^{-1}(t) := \inf\{s>t, L(s) >t\}.$$
There exists a Brownian motion $M$ such that for $t<L(\infty),$
$$M_t=\int_0^{L^{-1}(t)} H^0_s \de W_s.$$
Introduce the time-changed process
$$ (Y_t,\dot Y_t) = (X_{L^{-1}(t)}, \dot X_{L^{-1}(t)}),$$
stopped at time $L(\infty)$. In order to simplify the redaction, we
will often omit to specify ``stopped at time $L(\infty)$''. This
time change induces that the process $(Y,\dot Y)$ also does not
spend any time at zero, and that its excursions are that of $(X,\dot
X)$ belonging to $\E_2$, and stopped at $\Tr_1$ the first return
time to $\{0\} \times \Real$.
\begin{lemma} \label{SOR_zero}
The triplet $(Y_t,\dot Y_t, M_t)_{t\le L(\infty)}$ under $\Prob$ is
a solution of the equations $(SOR)$ with null elasticity
coefficient, stopped at time $L(\infty)$.
\end{lemma}
\begin{proof}
Let $[t,t'[$ be the interval corresponding to an excursion of
$(Y,\dot Y)$. Then the interval $[L^{-1}(t), L^{-1}(t' -)]$ is a
maximal interval included in $F$. It follows that the points
$L^{-1}(t)$ and $L^{-1}(t')$ belong to $\mathcal Z$, and $Y_t=\dot
Y_t= 0= Y_{t'}= \dot Y_{t'}$.

Let $s \in [t,t'[$. As the process $X$ has no bounce in
$[L^{-1}(t),L^{-1}(s)]$ and $(X,\dot X, W)$ is a solution to
$(SOR)$, we can write
$$ \dot X_{L^{-1}(s)}= \dot X_{L^{-1}(t)}+ W_{L^{-1}(s)} - W_{L^{-1}(t)},$$
or equivalently
$$ \dot Y_s=\dot Y_t + M_s - M_t.$$
As a consequence, we may write
$$\left\{ \begin{array}{ccl}
Y_s &=& Y_t + \displaystyle \int_t^s \dot Y_u\de u \\
\dot Y_s &=& \dot Y_t+ M_s-M_t - \sum_{t< u \le s} \dot Y_{u-}
\un_{Y_u=0},
\end{array}
\right. $$ where the sum is actually empty. Similarly,
$$\left\{ \begin{array}{ccl}
Y_{t'} &=& 0 = X_{L^{-1}(t' -)} = Y_{t' -} = Y_t + \displaystyle \int_t^{t'} \dot Y_u\de u \\
\dot Y_{t'} &=& 0 = \dot Y_{t'-} - \dot Y_{t'-} \un_{Y_{t'}=0}= \dot
Y_t+ M_{t'}-M_t - \sum_{t< u \le t'} \dot Y_{u-} \un_{Y_u=0},
\end{array}
\right. $$ where the sum now contains one term.

Adding these equalities on the excursion intervals of $(Y,\dot Y)$,
and recalling that this process spends no time at $(0,0)$, gives
$$\left\{ \begin{array}{ccl}
Y_s &=& \displaystyle \int_0^s \dot Y_u\de u \\
\dot Y_s &=& M_s - \sum_{0< u \le s} \dot Y_{u-} \un_{Y_u=0},
\end{array}
\right. $$ and $(Y, \dot Y, M)$ is a solution to $(SOR)$ with null
elasticity coefficient (stopped at time $L(\infty)$).
\end{proof}
The article \cite{SDE}, which studied equations $(SOR)$ with null
elasticity coefficient, shows that a solution $(Y,\dot Y)$ must be a
Markov process, with It\=o excursion law $\nbis$.
 We immediately introduce another change of time, in a very similar way, but without stopping the excursions of $\E_2$ at time $\Tr_1$.
Define the random set
$$ A := \mathcal Z \cup \bigcup_{\{t|e^t \in \E_2\}}[t,t+\Tri(e^t)],$$
and the adapted process $ \Ti H = \un_A.$ Define also
$$ \Ti L(t) = \int_0^t \Ti H_s \de s,$$
and $\Ti L ^{-1}$ for its right-continuous inverse. Then, there
exists a Brownian motion $\Ti M$ such that
$$ \Ti M_t = \int_0^{\Ti L ^{-1}(t)} \Ti H_s \de W_s$$
for $t <\Ti  L(\infty)$. Finally, the time-changed process
$$ (\Ti Y_t, \dot {\Ti Y}_t) = (X_{\Ti L^{-1}(t)}, \dot X_{\Ti L^{-1}(t)}),$$
stopped at time $\Ti L(\infty)$, spends no time at zero and its
excursions are the excursions of $(X,\dot X)$ included in $\E_2$.
 Remark that we have $\Ti L(\infty)
\ge  L(\infty)$ because $A \supset F$. We also get the following
lemma, similar to Lemma~\ref{SOR_zero}, and whose proof we leave to
the reader.
\begin{lemma} %\label{SOR_zero}
The triplet $\Big(\Ti Y_t,\dot {\Ti Y}_t, \Ti M_t\Big)_{t\le \Ti
L(\infty)}$ under $\Prob$ is a solution of the equations $(SOR)$
(with elasticity coefficient $c$), stopped at time $\Ti L(\infty)$.
\end{lemma}

The process $ (\Ti Y, \dot {\Ti Y})$ spends no time at $0$, is a
solution to $(SOR)$, and its excursions, stopped at $\Tr_1$, the
first return time to  $\{0\} \times \Real$, are precisely that of
$(Y,\dot Y).$ This induces that $ (\Ti Y, \dot {\Ti Y})$ is a Markov
process with It\=o excursion measure $\nti$ determined by

$$\left\{ \begin{array}{ccl}
\nti\left((x_{t \wedge \Tr_1})_{t \ge0} \in \cdot \right) &=& \nbis(x \in \cdot) \\
\nti\left((x_{t + \Tr_1})_{t \ge0} \in \cdot \right | \dot X_{\Tr_1}
= v) &=&  \Prob_v^c(x \in \cdot)
\end{array}
\right. $$ Now, the result of uniqueness of the excursion measure
implies that $\nti$ should be a multiple of $\n$, which is obviously
not the case (for example because $\nti(\Tri=0)=0$). Therefore $\Ti
L(\infty)=0= L(\infty)$ a.s. Lemma~\ref{lemma_F_mesure_nulle} is
proved.

\espace \espace

Now, introduce a third time-change, $(L^\eps)^{-1} (t) := \inf\{s>0,
L^\eps (s) >t\}$. When $\eps$ goes to 0, $(L^\eps)^{-1}$ is going to
$L^{-1}= Id$. It follows that the process $X^\eps :=
(X_{(L^\eps)^{-1}(t)})_{t\ge0}$ is going uniformly on compacts to
$X$ when $\eps$ goes to 0, almost surely. In particular the law of
$X$ is entirely determined by that of $X^\eps$. The law of $X^\eps$
is in turn entirely determined by that of $(\dot
X_{T_{2n+1}^\eps})_{n\ge 0}.$ We will now determine this law, which
will prove the uniqueness of the law of $X$.

\espace In order to avoid complex notations, we just give the
calculation of the law of $\dot X_{T_1^1}$, which is not
fundamentally different from others. For $\eps>0$ and $n\ge0$, a
Markov property for the process $W$ applied at time $T_{2n+1}^\eps$
shows that conditionally on $\dot X_{T_{2n+1}^\eps}=u$, the process
$(X_{(T_{2n+1}^\eps+t)\wedge T_{2n+2}^\eps})_{t\ge0}$ is independent
from $\F_{T_{2n+1}^\eps}$ and has law $\Prob_u^c$. Write $n_1$ for
the integer satisfying $T_{2n_1+1}^\eps \le T_1^1 <
T_{2n_1+2}^\eps$. Conditionally on $\dot X_{T_{2n_1+1}^\eps}=u$, the
process $(X_{(T_{2n_1+1}^\eps+t)\wedge T_{2n_1+2}^\eps})_{t\ge0}$
has the law $\Prob_u^c$ conditioned on reaching a speed greater than
one after a bounce.

In other words, the law of $\dot X_{T_1^1}$ under $\Prob( \cdot |
\dot X_{T_{2 n_1+1}^\eps} = u)$ is equal to that of $\dot X_{T_1^1}$
under $\Prob_u^c(\cdot | T_1^1<\infty)$. Besides, it should be clear
now that $\dot X_{T_{2 n_1+1}^\eps}$ is going to 0 when $\eps$ goes
to 0. Recall that $\Tri$, the hitting time of $(0,0)$, is the
lifetime of the excursion (under $\Prob_u^c$ as well as under $\n$).
For any $f$ positive continuous functional, we have:
\begin{eqnarray*}
  \Prob_u^c(f(\dot X_{T_1^1})|\ T_1^1< \Tri) &=& \Prob_u^c\big(f(\dot X_{T_1^1})\un_{T_1^1< \Tri}\big)\ /\: \Prob_u^c(\un_{T_1^1< \Tri}) \\
   &=& \Probt_u\Big(f(\dot X_{T_1^1}) (H(0,\dot X_{T_1^1}))^{-1}\Big)\ /\: \Probt_u((H(0,\dot X_{T_1^1}))^{-1} ) \\
   &\underset{u \to 0} \To&  \Probt_{0+}\big(f(\dot X_{T_1^1}) (H(0,\dot X_{T_1^1}))^{-1}\big)\ /\: \Probt_{0+}((H(0,\dot X_{T_1^1}))^{-1} ) \\
   &=& \n(f(\dot X_{T_1^1}) |\ T_1^1<\Tri),
\end{eqnarray*}
where we used successively \eqref{eq_def_Pup},
Proposition~\ref{proposition_theoRLP} and (a generalization of)
\eqref{formula1}. As a consequence, the law of $\dot X_{T_1^1}$
under $\Prob$ is entirely determined, and is equal to that of $\dot
X_{T_1^1}$ under $\n( \cdot |\ T_1^1< \Tri)$. Uniqueness of the
stochastic partial differential equation follows.

\bibliographystyle{abbrv}
\bibliography{jacob_RLP2_biblio}
\end{document}